\documentclass{amsart}
\usepackage[foot]{amsaddr}
\usepackage{geometry}
\usepackage{float}
\usepackage{latexsym}
\usepackage{amsthm,xcolor}
\usepackage{filecontents}
\newtheorem{thm}{Theorem}[section]

\theoremstyle{plain}
\newtheorem{lem}[thm]{Lemma}
\usepackage{amsmath}
\usepackage{graphicx}
\usepackage{subfigure}
\usepackage{amssymb}
\usepackage{fullpage}
\usepackage{amsthm}
\newtheorem{asu}[thm]{Assumption}
\newtheorem{defn}[thm]{Definition}
\usepackage{comment}
\usepackage{appendix}
\usepackage{cleveref}
\newtheorem{remark}[thm]{Remark}
\usepackage{amssymb}
\usepackage{footmisc}
\usepackage[displaymath, mathlines]{lineno}

\title{Asymptotic Properties of Linear Filter for Noise Free Dynamical System}

\author{Anugu Sumith Reddy}  \email{anugu.reddy@icts.res.in}
\address{International Centre for
  Theoretical Sciences - Tata Institute of Fundamental Research,
  Bangalore, India}
\author{Amit Apte} \email{apte@icts.res.in}
\address{International Centre for
  Theoretical Sciences - Tata Institute of Fundamental Research,
  Bangalore, India}

\author{Sreekar Vadlamani} \email{sreekar@tifrbng.res.in}
\address{TIFR Centre for Applied
  Mathematics, Bangalore, India}

\begin{document}

\begin{abstract}
  It is known that Kalman-Bucy filter is stable with respect to
  initial conditions under the conditions of uniform complete
  controllability and uniform complete observability
  \cite{bishop2017stability, ocone1996asymptotic}. In this paper, we
  prove the stability of Kalman-Bucy filter for the case of noise free
  dynamical system. The earlier stability results cannot be applied
  for this case, as the system is not controllable at all. We further
  show that the optimal linear filter for certain class of
  non-Gaussian initial conditions is asymptotically proximal to
  Kalman-Bucy filter. It is also shown that the filter corresponding
  to non-zero system noise in the limit of small system noise
  approaches the filter corresponding to zero system noise in the case
  of Gaussian initial conditions.
\end{abstract}

\keywords{Kalman-Bucy; Noise free; Stability; Small noise limit}

\maketitle

\section{Introduction}
\label{S1}

Since the seminal paper of Kalman and Bucy \cite{kalman1961new},
Kalman-Bucy filter is extensively studied
\cite{anderson1979optimal,jazwinski2007stochastic}. It gives the best
mean square estimate of the state at a fixed time $t$, given the
observations up to time $t$, when the dynamical system and the
observation model are linear and the initial condition is
Gaussian. Studying asymptotic properties of filters with respect to
initial conditions of the filter is an important aspect of filtering
theory, primarily to unravel certain universalities among different
filters, and to gain some understanding into the large time behaviour
of the filters. In practice, the exact initial condition of the system
is rarely known. Therefore, it is desirable that the filter be
asymptotically independent of initial condition. This property, known
as filter stability, has been studied extensively
\cite{ocone1996asymptotic, bishop2017stability}.

The classical results on stability of Kalman-Bucy filter are based on
the assumption of controllability. If the system being observed is
modelled by a deterministic process (in other words, zero system noise
in case of additive noise systems), the assumption of controllability
breaks down and the classical results are not applicable, and a new
approach is needed. This problem was first studied in
\cite{ni2016stability} while for nonlinear systems, the asymptotic
convergence of the filter estimate to the true state in the case of
zero system noise is studied in \cite{cerou2000long}.

In practice, filtering for deterministic systems is quite commonly
used in the context of atmospheric and oceanic sciences where the
problem is known as data assimilation. \cite{sarkka2013bayesian,
  asch2016data, fletcher2017data} In these applications, the
asymptotic degeneracy and stability of the filter covariance (but not
of the filter mean) for discrete time Kalman filter has been studied
recently \cite{gurumoorthy2017rank, bocquet2017degenerate} and
generalising those results to filter stability for continuous time
Kalman-Bucy filters is one of the main aim of this work.

The unifying theme of this work is to study stability of linear
filters in the following three cases: (a) Kalman-Bucy filter in the
case of zero system noise; (b) linear filter with non Gaussian initial
conditions; and (c) linear filter with small system noise case and
examine its relation with zero system noise case.

The methods used in this work are motivated from the results of Ocone
and Pardoux \cite{ocone1996asymptotic} on the stability of Kalman-Bucy
filter. However, analogous results, for the case of zero system noise,
do not follow trivially as the crucial assumption of stabilizability
becomes invalid. As mentioned earlier, Ni and Zhang
\cite{ni2016stability} studied this problem and proved the stability
of Dynamic Riccati Equation (see \eqref{DRE} below), whereas we show
(in theorems~\ref{mainth}-\ref{thm:gaussianmean}) a stronger result
that the filter initialised with incorrect initial condition converges
asymptotically to the optimal filter almost surely. We also show (in
theorem~\ref{thm:nongaussian}) that even with non-Gaussian initial
conditions, the optimal filter approaches the Kalman-Bucy filter. It
is also shown (in theorem{thm:noisy}) that, under appropriate
assumptions, the case with small system noise is asymptotically
similar to the case with zero system noise.
   
The paper is organised as follows: the main setup and statement of the
problem is introduced in Section \ref{S2}. Thereafter, in Section
\ref{S3} we study the asymptotic properties of filter in the case of
Gaussian initial conditions. The case of linear filter with
non-Gaussian initial conditions is discussed in Section \ref{S4}. In
particular, we establish that for a particular class of non-Gaussian
initial conditions, the optimal filter asymptotically approaches the
Kalman-Bucy filter. Finally, in Section \ref{S5}, it is also shown
that small noise limit of the filter corresponding to non-zero system
noise is indistinguishable to the filter corresponding to zero system
noise in the case of Gaussian initial conditions.

\section{Problem Setup}\label{S2}

Let $(\Omega, \mathcal{F},\{\mathcal{F}_t\}_{t\geq 0},\mathbb{P})$ be
a complete filtered probability space satisfying usual conditions,
\emph{i.e}, $\mathcal{F}_0$ contains all $\mathbb{P}$-null sets and
$\mathcal{F}_t$ is right continuous. We consider the following
filtering model for a linear signal process $x_t \in \mathbb{R}^{m}$,

\begin{align}\label{eq1}
  x_t&=x_{0}+\int_{0} ^t A_sx_sds,
\end{align}
with linear observation process $y_t \in \mathbb{R}^{n}$,

\begin{align} \label{obmod}
  y_t&=\int_{0} ^tC_sx_sds+\int_{0} ^t
  R^{\frac{1}{2}}_s dW_s\,,
\end{align}
where, $t\geq 0$, $A_t\in \mathbb{R}^{m\times m}$,
$C_t \in \mathbb{R}^{n\times m}$ and $R_t \in \mathbb{R}^{n\times
  n}$. Let $\mathcal{Y}_t:=\sigma(y_s:0\leq s\leq t)$ be the
$\sigma$-field generated by the observation process and $W_t$ be the
$m$-dimensional $\mathcal{F}_t$-standard Brownian motion independent
of $x_{0}$. The central theme of interest in filtering theory is
estimating $x_t$, given the observations up to time $t$, which is
calculating $\mathbb{E}[x_t|\mathcal{Y}_t]$. Since we are usually
interested in estimating functions of $x_t$, we are interested in the
calculating conditional distribution,
$\pi_t(B):=\mathbb{E}[\mathsf{1}_{ x_t\in B}|\mathcal{Y}_t]$, where
$B \in \mathbb{B}(\mathbb{R}^m)$.

We now introduce Kalman-Bucy filtering equations which play a crucial
role in the rest of the paper. These are given by,

\begin{align}\label{eq2}
  dX^{m,P}_t&=A_tX^{m,P}_tdt+P^P_t C^T _t R^{-1}_t(dy_t-C_tX^{m,P}_tdt),\;\; X^{m,P}_0=m \in \mathbb{R}^{m},\\\label{DRE}
  \dot{P}^P_t&= A_tP^P_t +P^P_tA^T _t -P^P_tC^T _t R^{-1} _t C_tP^P_t,\;\; P^P_0=P \in \mathbb {R}^{m\times m},
\end{align}
where\footnote{For real symmetric positive semi-definite matrices $X$
  and $Y$ of same dimension, we write $X \geq Y$ whenever
  $\textbf{x}^T (X-Y)\textbf{x}\geq 0, \forall\; \textbf{x}\neq
  \textbf{0} \in \mathbb{R}^m$. Notations like $X\leq Y$, $X< Y$ and
  $X> Y$ are adopted accordingly throughout the paper.} $P>0$. Note
that the superscripts in $X^{m,P}_t$ and $P^{P}_t$ refers to initial
conditions of \eqref{eq2} and \eqref{DRE}. It is well known
\cite[Theorem~9.4]{xiong2008introduction} that, for the filtering
model in \eqref{eq1} and \eqref{obmod}, if the initial condition is
Gaussian, $x_0 \sim \mathcal {N}(m_0,P_0)$, then the conditional
distribution $\pi_t$ is Gaussian,
$\pi_t \sim \mathcal {N}(\hat{X}_t,P_t)$, with mean
$\hat{X}_t:=X^{m_0,P_0}_t$ and covariance $P_t:=P^{P_0}_t$. It is
clear that the Gaussian distribution
$\bar{\pi}_t := \mathcal{N}(X^{\bar{m},\bar{P}}_t,P^{\bar{P}}_t)$,
where $X^{\bar{m},\bar{P}}_t$ and $P^{\bar{P}}_t$ are solutions of
\eqref{eq2} and \eqref{DRE} with initial conditions
$(\bar{m},\bar{P})$ different from $(m_0,P_0)$, is not the same as
$\pi_t$. The linear filter is said to be stable with respect to
initial conditions if
$(\pi _t - \bar{\pi}_t) \xrightarrow{t\to \infty} 0$ with respect to
an appropriate metric. We will discuss such stability results for both
Gaussian and non-Gaussian initial conditions.

We shall make the following assumptions throughout the paper.
\begin{asu} 
  $A_t$,$C_t$,$R_t$, $R^{-1}_t$ are all continuous and uniformly
  bounded in $t$ and $P_0$ is invertible.
\end{asu}
In order to state the next assumption, we first need the following
definition \cite{anderson1969new}.
 \begin{defn}
   A pair $[A_t,C_t]$, $A_t\in \mathbb{R}^{m\times m}$,
   $C_t \in \mathbb{R}^{n\times m}$ is said to be \emph{uniformly
     completely observable}, if there exist positive constants $\tau$,
   $\rho_1$, $\rho_2$, such that for all $t \geq 0$, we have

\begin{align}\label{obs}
  \rho_1 I_n\leq \int_{t-\tau}^{t} \Phi^{-T}_t\Phi^{T} _s C^T_s R^{-1}_s C_s\Phi_s\Phi^{-1} _t ds \leq \rho_2 I_n.
\end{align}
\end{defn}
Here, ${\Phi}_t$ is the fundamental matrix solution of \eqref{eq1},
$\emph{i.e}$, $\dot{\Phi}_t=A_t \Phi_t$ and ${\Phi}_0:=\mathbb{I}$.
Additionally, we also assume that:
\begin{asu}\label{uco}
  The pair $[A_t,C_t]$ is uniformly completely observable.
\end{asu}

\section{Asymptotic properties of filter in case of Gaussian initial
  conditions} \label{S3}
In this section, we study the asymptotic properties and stability of
the filter when the initial condition is assumed to be Gaussian. As
mentioned earlier, we are interested in calculating the distance
between measures $\pi_t$ and $\bar{\pi}_t$ under appropriate
metric. Since both $\pi _t$ and $\bar{\pi}_t$ are Gaussian, if we
choose the total variation metric, then showing the convergence of
$\|\hat{X}_t-X^{\bar{m},\bar{P}}_t \| \xrightarrow[] {t \to \infty}0$
and $\|P_t-P^{\bar{P}}_t \| \xrightarrow[] {t \to \infty}0$ is
sufficient to establish the stability of the filter for the Gaussian
initial condition. In other words, it is sufficient to prove the
stability of \eqref{DRE} which is called the Dynamic Riccati Equation
and of \eqref{eq2}. Throughout this paper, we define the norm
$\|\cdot\|$ of a $m\times n$ matrix $Q$ as
$\|Q\|:=\sup_{\|x\|=1} \|Qx \|$.

\subsection{Stability of the dynamic Riccati equation}

We begin with observing that the solution of \eqref{DRE} with a
non-negative definite initial condition $P ge 0$ can be written as

\begin{align}\label{DREsol}
  P^P_t=\Phi_t\sqrt{P}\Big(I+\sqrt{P} \bar{C}_t \sqrt{P}\Big)^{-1}\sqrt{P} \Phi^T_t,
\end{align}
where
$\bar{C}_t:= \int_{0} ^{t} {\Phi}^T _s C^T_s R^{-1} _s C_s {\Phi} _s
ds$. To investigate the stability of \eqref{DRE}, we need
the following result proved in \cite{ni2016stability} which concerns
the uniform boundedness of $P^P_t$.
\begin{lem}
  If $[A_t,C_t]$ is uniformly completely observable, $P^P_t$ is
  uniformly bounded in $t$.
\end{lem}
\begin{remark}
  Consider the subspace of $\mathbb{R}^m$ defined by
  $S := \{ u: \|\Phi^T _t u\|\rightarrow 0 \;\; as\;\; t\rightarrow
  0\}$.  For $\textbf{v} \in S$, it is clear from \eqref{DREsol} that
  $\textbf{v}^T P^P_t \textbf{v} \rightarrow 0$ as
  $t \rightarrow \infty$ (since $\bar{C}_t$ is bounded below uniformly
  in time \cite[Proposition 3]{ni2016stability}), implying that the
  uncertainty along $S$ reduces to zero asymptotically in time. This
  feature is used in data assimilation algorithms in discrete time
  that go by the name of Assimilation in Unstable Subspace (AUS)
  \cite{carrassi2008controlling, palatella2013lyapunov,
    trevisan2011kalman}. This and other properties of the filter
  covariance (in discrete time) and their relation to Lyapunov vectors
  and exponents of the dynamics that have been discussed extensively
  in \cite{gurumoorthy2017rank, bocquet2017degenerate} extend to the
  filter covariance for the Kalman-Bucy filter (in continuous time).
\end{remark}
To prove stability of \eqref{DRE}, we consider solutions $P^P _t$ and
$P^{\bar {P}}_t$ of \eqref{DRE} corresponding to two different initial
conditions $P$ and $\bar{P}$, respectively. A straightforward
calculation shows that $E_t:= P^P_t-P^{\bar {P}}_t$ satisfies

\begin{align*}
  \dot{E}_t&= B^P _tE_t +E_t\big({B}^{\bar{P}}_t\big)^T ,
\end{align*}
where $B^P _t:=\big(A_t-P^P_tC^T _t R^{-1} _t C_t\big)$ and
${B}^{\bar{P}}_t:=\big(A_t-P^{\bar{P}}_tC^T _t R^{-1} _t
C_t\big)$. Further, it can easily be verified that

\begin{align}\label{errest}
E_t={\Psi}^P_t\big(P-\bar{P}\big)\big(\Psi ^{\bar{P}}_t\big)^T,
\end{align}
with $\dot{{\Psi}}^{P}_t=B^P _t {\Psi}^P _t$,
${\Psi}^P _0=\mathbb{I}$,
$\dot{{\Psi}}^{\bar{P}}_t=B^{\bar{P}}_t \Psi^{\bar{P}}_t$ and
$\Psi^{\bar{P}}_0=\mathbb{I}$. Therefore, stability of the Riccati
equation is related to studying the asymptotic properties of
$\Psi^{{P}}_t$ and $\Psi^{\bar{P}}_t$. Without loss of generality, it
is sufficient to study asymptotic properties of $\Psi^{{P}}_t$. To
this end, consider a linear system

\begin{align}\label{hom}
  \dot{z}_t=\big(A_t-P^P_t C_t^T R^{-1}_tC_t\big)z_t\,,
\end{align}
whose solution is given by $z_t=\Psi^{{P}}_t z_0$, where $z_0$ is the
initial condition. The above system \eqref{hom} is said to be
asymptotically stable if
$\|\Psi^{{P}}_t\|\xrightarrow[] {t \to \infty}0$ which is equivalent
to
$\|z_t\|\xrightarrow[] {t \to \infty}0, \;\forall\; z_0 \in
\mathbb{R}^m$. Therefore, to establish that
$\|\Psi^{{P}}_t\|\xrightarrow[] {t \to \infty}0$, we use Lyapunov
function approach used in \cite{bucy1967global} and show that
$\|z_t\|\xrightarrow[] {t \to \infty}0, \;\forall\; z_0 \in
\mathbb{R}^m$. The first step towards proving asymptotic stability of
\eqref{hom}, is the following lemma
\cite[Lemma~2.5.2]{sastry2011adaptive}.

\begin{lem}\label{ucoo}
  If $[A_t,C_t]$ is uniformly completely observable and $K_t$ is
  continuous and bounded in $t$, then $[A_t-K_tC_t,C_t]$ is also
  uniformly completely observable.
\end{lem}  
Consequently, since $P^P_tC^T _t R^{-1} _t$ is continuous and bounded
in $t$, $[B^P_t, C_t]$ is uniformly completely observable, i.e., there
exist $\tilde{\tau},\rho_3,\rho_4>0$ such that for all $t>0$ we have,

\begin{align}\label{obmean}
  \rho_3 I_n\leq \int_{t-\tilde{\tau}}^{t}  \big(\Psi^P_t\big)^{-T} \big(\Psi^P_s\big)^T C^T_s R^{-1}_s C_s\Psi^P_s \big(\Psi^P_t\big)^{-1} ds \leq \rho_4 I_n\,.
\end{align} 
We shall now state one of the main results of this paper, that of
asymptotic stability of the filter covariance.
\begin{thm}\label{mainth}
  Let $P$ be non-negative definite, and $[A_t,C_t]$ be uniformly
  completely observable, then \eqref{hom} is asymptotically stable and

\begin{align}\label{eqn:Psi-decay}
  \int_0 ^{\infty} \big(\Psi^P_s \big)^T\Psi^P_sds < \frac{\tilde{\tau} P^{-1}}{\rho_3}
\end{align}
\end{thm}
\begin{remark}
  We note here that the asymptotic stability of \eqref{hom} has
  already been proven in \cite{ni2016stability} using Lyapunov
  function, wherein it is shown that
  $\|z_t\|\xrightarrow[] {t \to \infty}0$, which in turn implies the
  stability of the Riccati equation. However, our result above is
  stronger since \eqref{eqn:Psi-decay} gives certain control over the
  rate of decay of $\Psi^P_s$, which is needed later to prove almost
  sure convergence of the filter mean.
\end{remark} 
\begin{proof} 
  Like in \cite{ni2016stability}, we begin with a Lyapunov function
  
\begin{align}\label{lyap}
    V(z_t,t):= z_t ^T \big(P^P_t\big)^{-1}z_t.
  \end{align}
  Using \eqref{DRE} and \eqref{hom}, we see that

\begin{align}\nonumber
  \frac{dV}{dt}(z_t,t)&=-z_t^T(A_t-{P^P}_tC^T _tR^{-1}_t C_t)^T {\big(P^P_t\big)}^{-1} z_t\\\nonumber &+z^T_t(-A^T_t{\big(P^P_t\big)}^{-1}-{\big(P^P_t\big)}^{-1}A_t+C^T_t R^{-1} _t C_t)z_t\\\nonumber &+z^T_t{\big(P^P_t\big)}^{-1}(A_t-{P^P}_tC^T _tR^{-1}_t C_t)z_t\\ \label{lyapineq}
&=-z_t^T C^T_t R^{-1} _t C_tz_t \leq 0 \,, \qquad \forall t> 0 \,.
\end{align}
Using the relationship $z_s=\Psi^P_s \big(\Psi^P\big)^{-1}_t z_t$,
we can write

\begin{align}\nonumber
  V(z_{t+\tilde{\tau}},t+\tilde{\tau})-V(z_t,t)&=-z_t^T\int_{t}^{t+\tilde{\tau}} \big(\Psi^P_t \big)^{-T}  \big(\Psi^P_s\big)^TC^T_s R^{-1} _s C_s\Psi^P_s\big(\Psi^P _t\big)^{-1}  ds\: z_t\,.
\end{align} 
Observe that from \eqref{obmean},

\begin{align}\label{meanin}
  \rho_3\|z_t\|^2\leq V(z_t,t)- V(z_{t+\tilde{\tau}},t+\tilde{\tau})&\leq\rho_4\|z_t\|^2,
\end{align}
which together with the assumption of uniform complete observability
of $[A_t,C_t]$ imply that $V(z_t,t)\rightarrow 0$, and
$\|z_t\| \rightarrow 0$, as $t\rightarrow \infty$, and that
\eqref{hom} is asymptotically stable.

Next, in order to prove \eqref{eqn:Psi-decay}, observe that writing
$t=t' +k\tilde{\tau},$ for some $t'\in[0, \tilde{\tau}]$, we have

\begin{align*}
  V(z_{t'+(k+1)\tilde{\tau}},t'+(k+1)\tilde{\tau})-V(z_{t+k\tilde{\tau}},t' +k\tilde{\tau})&\leq-\rho_3\|z_{t' +k\tilde{\tau}}\|^2
\end{align*}
Adding $N$ such inequalities with $k=0,1,2,...,N$, we have

\begin{align*}
  V(z_{t'+(N+1)\tilde{\tau}},t'+(N+1)\tilde{\tau})-V(z_{t'},t')&\leq-\rho_3\sum_{k=0}^{N}\|z_{t' +k\tilde{\tau}}\|^2
\end{align*}
Using \eqref{lyapineq}, and letting $N\rightarrow \infty$,

\begin{align}\label{sum}
  \sum_{k=0}^{\infty}\|z_{t' +k\tilde{\tau}}\|^2 \leq \frac{V(z_{t'},t')}{\rho_3} \leq \frac{V(z_{0},0)}{\rho_3} \,.
\end{align}
Integrating \eqref{sum} with respect to $t'$ in the range
$t'\in[0,\tilde{\tau}]$, we have

\begin{align}\nonumber
  \int_{0}^{\tilde{\tau}}\sum_{k=0}^{\infty}\|z_{t' +k\tilde{\tau}}\|^2dt' &\leq \int_{0}^{\tilde{\tau}}\frac{V(z_{0},0)}{\rho_3}dt'\\\nonumber
  \int_{0}^{\infty}\|z_{t'}\|^2dt' &\leq \frac{V(z_{0},0)\tilde{\tau}}{\rho_3}\\ 
  z_{0}^T\Big(\int_{0}^{\infty}\big(\Psi^P_{t'}\big)^T\Psi^P_{t'}dt'\Big)z_{0} &\leq \frac{V(z_{0},0)\tilde{\tau}}{\rho_3}= \frac{\tilde{\tau}z^T _{0}P^{-1} z_{0}}{\rho_3}
\label{bou}
\end{align}
Since \eqref{bou} is true for all initial conditions $z_{0}$,

\begin{align}\label{integrability}
  \int_{0}^{\infty}\big(\Psi^P _{t'}\big)^T\Psi^P_{t'}dt'&\leq  \frac{\tilde{\tau}}{\rho_3}P^{-1}
\end{align}
which completes the proof.
\end{proof}

\subsection{Almost sure convergence of the conditional expectation}

To discuss the convergence of conditional expectation, we follow the
method set forth in \cite{ocone1996asymptotic}. Consider two solutions
$(\hat{X} _t,P_t)$, $(X^{\bar{m},\bar{P}} _t, P^{\bar{P}}_t)$ of
\eqref{eq2} and \eqref{DRE} with different initial conditions: one
correct, $m_0,P_0$ (which are the mean and covariance of the Gaussian
$x_0$) and the other incorrect, $\bar{m},\bar{P}$, respectively. Our
result concerning the asymptotic stability of the filter mean is as
follows:
\begin{thm}
  Let $P_0, \bar{P}$ be some bounded non-negative definite matrices,
  and $[A_t,C_t]$ be uniformly completely observable, then
  $\|\hat{X} _t- X^{\bar{m},\bar{P}} _t \| \xrightarrow[]{t\to \infty}
  0 \;\; \mathbb{P}-a.s$
\label{thm:gaussianmean} \end{thm}
\begin{proof}
  Let us begin with defining the innovations process
  \[ d\nu_t:= dy_t-C_t\hat{X}_tdt, \] which is a
  $\mathcal{Y}_t$-Brownian motion \cite{xiong2008introduction}. Then,
  the using \eqref{eq2}, we see that the dynamical equation for
  $\hat{X} - X^{\bar{m},\bar{P}}$ is
  
\begin{align}
    d(\hat{X} _t-X^{\bar{m},\bar{P}} _t)&=\big(A_t-P^{\bar{P}}_t C^T _t R^{-1}_t C_t\big)(\hat{X} _t-X^{\bar{m},\bar{P}} _t)dt+\big(P_t-P^{\bar{P}}_t\big)C^T_t R^{-1}_t (dy_t-C_t\hat{X}_tdt) \,. \label{err}
  \end{align}
  Using a simple application of Ito's formula, we observe that
  solution to he above dynamical equation is given by
  
\begin{align*}
    \big(\hat{X}-X^{\bar{m},\bar{P}}\big)_t=\Psi^{\bar{P}}_t(m_0-\bar{m})+\int_0 ^{t} \Psi^{\bar{P}}_t \big(\Psi^{\bar{P}}_s\big)^{-1}\big(P_s-P^{\bar{P}}_s\big)C^T_s R^{-1}_s d\nu_s
  \end{align*}
  Next, writing
  $\hat{Z}_t:= \int_0 ^{t}
  \big(\Psi^{\bar{P}}_s\big)^{-1}\big(P_s-P^{\bar{P}}_s\big)C^T_s
  R^{-1}_s d\nu_s$, we can express the above solution in a compact
  form as
  
\begin{align}\label{eqn:mean-decomposition}
    \big(\hat{X}-X^{\bar{m},\bar{P}}\big)_t=\Psi^{\bar{P}}_t(m_0-\bar{m})+ \Psi^{\bar{P}}_t \hat{Z}_t. 
  \end{align}
  Observe now that using \eqref{errest} to write $(P_s-P^{\bar{P}}_s)$
  in terms of $\Psi^{P_0}_s$ and $\Psi ^{\bar{P}}_s$, it is clear
  that,
  
\begin{align}
    \mathbb{E}[|\hat{Z}_t|^2]&=\mathbb{E}\left[\textrm{tr}\Big(\int_0 ^t (P_0-\bar{P})\big(\Psi^{P_0}_s\big)^T C^T_s R^{-1}_sR^{-1}_sC_s
                               \Psi^{{P}_0}_s(P_0-\bar{P})ds\Big)\right],
  \end{align}
  where $\textrm{tr}(A)$ denotes the trace of the square matrix $A$.
  Using simple algebra, we can easily conclude that for some $M'>0$,
  we have $(P_0-\bar{P})^2<M' I$. In particular, we could choose $M'$
  to be the squared sum of the largest eigenvalues of $P_0$ and
  $\bar{P}$. Moreover, we also have $\|C^T _t R^{-2} _t C_t\| <M$, for
  some $M>0$, thus implying
  
\begin{align*}
    \textrm{tr}\Big(\int_0 ^t (P_0-\bar{P})\big(\Psi^{P_0}_s\big)^TC^T_s R^{-1}_sR^{-1}_sC_s
    \Psi^{P_0}_s(P_0-\bar{P})ds\Big)&\leq M M' \textrm{tr}\Big(\int_0 ^t \big(\Psi^{P_0}_s\big)^T
                                      \Psi^{P_0}_sds\Big)\\
                                    &\leq M M' \textrm{tr}\Big(\int_0 ^\infty \big(\Psi^{P_0}_s\big)^T
                                      \Psi^{P_0}_sds\Big)\\
                                    &< \infty, 
  \end{align*}
  where the last inequality follows from Theorem \eqref{mainth},
  indicating that $\hat{Z}_t$ is a square integrable
  martingale. Therefore, by martingale convergence theorem,
  $\{\hat{Z}_t\}_{t\ge 0}$ converges almost surely, as $t\to\infty$,
  to an integrable random variable, say $N$. Thus, we conclude that
  $ \Psi^{\bar{P}}_t \hat{Z}_t \rightarrow 0 \;\;\mathbb{P}-a.s$,
  since we already know by Theorem \eqref{mainth} that
  $\Psi^{\bar{P}}_t$ converges to zero as $t\to\infty$. Similarly, we
  can deduce that $\Psi^{\bar{P}}_t(m_0 - \bar{m}) \to 0$, as
  $t\to\infty$, which in view of \eqref{eqn:mean-decomposition}
  completes the proof.
\end{proof}

\section{Linear filter with non-Gaussian initial condition}\label{S4}

In this section, we will consider the filter stability in the case of
non-Gaussian initial conditions. Note that if $x_0$ is not Gaussian,
then $\pi _t $ is not Gaussian either. But the following theorem shows
that, under certain conditions, the linear filter even with
non-Gaussian initial condition is asymptotically close to an
appropriate Kalman-Bucy filter almost surely. To state the theorem
below, we recall that $X^{\bar{m},\bar{P}}_t,P^{\bar{P}}_t$ denote the
solutions of \eqref{eq2} and \eqref{DRE} with initial conditions
$\bar{m}$ and $\bar{P}$.
\begin{thm}
  Suppose the pair $[A_t,C_t]$ is uniformly completely observable. Let
  $x_0$ be square integrable and be of the form
  $x_{0} :=v_0+\bar{x}_0$, where $\bar{x}_0$ is a non-degenerate
  Gaussian random variable independent of $v_0$. Then for the system
  given by \eqref{eq1} and \eqref{obmod}, the filter mean
  $\mathbb{E}[x_t|\mathcal{Y}_t]$ is almost surely asymptotically
  proximal to the filter mean $X^{\bar{m},\bar{P}}_t$ :

  \begin{align}\label{1s}
    \mathbb{E}[x_t|\mathcal{Y}_t]-X^{\bar{m},\bar{P}}_t\xrightarrow[]{t \to \infty} 0,\;\; \mathbb{P}-a.s.
  \end{align}
  We also have the almost sure weak asymptotic proximality (or
  merging, following the terminology of \cite{daristotile1988merging})
  of the filtering distributions $\pi_t$ with the Gaussian
  distributions defined by solutions of Kalman-Bucy equations:
  
  \begin{align}\label{2s}
    \pi_t(g)-\mathcal{N}(X^{\bar{m},\bar{P}}_t,P^{\bar{P}}_t)(g)\xrightarrow[] {t \to \infty} 0,\;\; \mathbb{P}-a.s, 
  \end{align}
  for any bounded, uniformly continuous $g$, for any
  $\bar{m}\in \mathbb{R}^m$ and $\bar{P} \in \mathbb{R}^{m \times m}$,
  $\bar{P}>0$.
\label{thm:nongaussian} \end{thm}
\begin{remark}
  The requirement that the initial condition $x_0$ be a sum of a
  Gaussian and a non-Gaussian random variables is not very
  restrictive. One quite large class of random variables that satisfy
  this assumption is as follows: for every $m$-dimensional random
  variable $U$ with finite second moment and a density $f_U$, there is
  a corresponding $x_0$ satisfying the assumptions of the theorem,
  where $x_0$ is defined to be a random variable with density which is
  a solution of $m$-dimensional heat equation initialised with $f_U$.
\end{remark}
\begin{proof}
  The ideas of our proof are motivated by \cite{makowski1986filtering}
  and by those used in the proof of \cite[Theorem
  2.6]{ocone1996asymptotic}, with certain modifications to accommodate
  our model with zero noise.
 
  First, observe that the system given by \eqref{eq1} and
  \eqref{obmod} can also be represented as
  
  \begin{align*}
    x_t&= \bar{x}_t +\Phi_tv_0,\;\; \bar{x}_t=\Phi_t\bar{x}_0\\
    y_t&=\int_{0} ^tC_s\bar{x}_sds+ W_t + \int_{0} ^tC_s\Phi_sv_0ds
  \end{align*}
  Here, $\bar{W}_t:=W_t + \int_{0} ^tC_s\Phi_s v_0ds$ is not a
  Brownian motion with respect to $\mathbb{P}$. Hence we invoke a
  change of measure transformation to find a new probability measure
  $\bar{\mathbb{P}}$, with respect to which $\bar{W}_t$ is a Brownian
  motion. By introducing such a transformation we can use much of the
  analysis related to Gaussian initial conditions with appropriate
  modifications. The authors in \cite{makowski1986filtering} and
  \cite{ocone1996asymptotic} use precisely this idea, to analyse the
  case of non-Gaussian initial conditions in their works.

  Let us begin with defining $Z_t$ for $t>0$ by

  \begin{align*}
    Z_t:=\exp\big(-\int_0 ^{t}(C_s\Phi_s v_0)^TdW_s-\frac{1}{2}\int_0 ^{t} \|C_s\Phi_s v_0\|^2 ds   \big),
  \end{align*}
  and define a new measure $\bar{\mathbb{P}}_{T_0}$ for some fixed
  $T_0>0$ by the Radon-Nikodym derivative

  \begin{align*}
    \frac{d\bar{\mathbb{P}}_{T_0}}{d{\mathbb{P}}}:=Z_{T_0}^{-1}
  \end{align*}
  where $\bar{\mathbb{P}}_{T_0}$ defined as above is a probability
  measure (by \cite[Corollary~3.5.16]{karatzas2012brownian}) and
  equivalent to ${\mathbb{P}}$ for all $T_0<\infty$. Notice that the
  variables $v_0$, $\bar{W}_t$, $\bar{x}_t$ are all mutually
  independent under $\bar{\mathbb{P}}_{T_0}$, and that the
  distribution of $v_0$ remains unchanged. Denoting the expectation
  with respect to $\bar{\mathbb{P}}_{T_0}$ by $\bar{\mathbb{E}}$, we
  see that the expectations with respect to the two probability
  measures are related by (\cite{xiong2008introduction}):

  \begin{align}\label{bayes}
    \mathbb{E}[f(v_0,\bar{x}_t)|\mathcal{Y}_t]=\frac{\bar{\mathbb{E}}[f(v_0,\bar{x}_t)Z_t|\mathcal{Y}_t]}{\bar{\mathbb{E}}[Z_t|\mathcal{Y}_t]}
  \end{align}
  for any bounded measurable function
  $f:\mathbb{R}^m \times \mathbb{R}^m \to \mathbb{R}$. Writing $\pi$
  for the distribution of $v_0$, it is easy see that
 
  \begin{align}\label{bp}
    \bar{\mathbb{E}}[f(v_0,\bar{x}_t)Z_t|\mathcal{Y}_t]=\int_{\mathbb{R}^m}\pi(dx)\int_{\mathbb{R}^{2m}}f(x,r_1) e^{-\frac{1}{2}x^TM_t x + x^T r_2} \eta_t (dr_1,dr_2),
  \end{align}
  where, $M_t:=\int_0 ^t \Phi_s^T C_s ^T C_s \Phi_s ds$,
  $b_t:=\int_0 ^t (C_s\Phi_s)^Td\bar{W}_s$, and $\eta_t$ is the
  conditional distribution of $ \begin{pmatrix}
    \bar{x}_t\\
    {b}_t
  \end{pmatrix}$ given $\mathcal {Y}_t$ under
  $\bar{\mathbb{P}}_{T_0}$. The conditional distribution $\eta_t$ is
  obtained by studying Kalman-Bucy filter in the framework with
  correlated observation and system noises for the extended system,
  $ \begin{pmatrix}
    \bar{x}_t\\
    {b}_t
  \end{pmatrix}$. It is known that the conditional distribution
  $\eta_t$ is again Gaussian \cite{xiong2008introduction}, with mean
  $ \begin{pmatrix}
    \tilde{m}_t\\
    \tilde{b}_t
  \end{pmatrix}$
  and  covariance 
  $ \begin{pmatrix}
    \tilde{P}_t & S_t \\
    S_t^T & Q_t
  \end{pmatrix}$
  given by the following set of equations.
  
  \begin{align}\nonumber
    \tilde{m}_t&=X^{m',P'}_t \textrm{      (solution of \eqref{eq2})},  &\tilde{m}_0 &= m' = \mathbb{E}[\bar{x}_0] \,,\\\nonumber
    \tilde{P}_t&=P^{P'}_t \textrm{      (solution of \eqref{DRE})},     &\tilde{P}_0 &= P' =\mathbb{E}[(\bar{x}_0-\mathbb{E}[\bar{x}_0])(\bar{x}_0-\mathbb{E}[\bar{x}_0])^T]\,,\\\nonumber
    d\tilde{b}_t&=(\Phi_t + S_t)^T C_t ^T (dy_t-C_t\tilde{m}_tdt)\,,   &\tilde{b}_0&=0\,,\\\nonumber
    \dot{Q}_t&=-\Phi_tC_t ^T C_t S_t -S_t ^T C_t ^TC_t \Phi_t- S_t ^T C_t ^T C_t S_t\,,    &Q_0 &=0\,,\\\label{S}
    \dot{S}_t&=A_t S_t-\tilde{P}_tC_t ^T C_t S_t - \tilde{P}_tC_t ^T C_t \Phi_t\,,   &S_0&=0\,.
  \end{align}    
  Back to computing expectations, we use \eqref{bp} in \eqref{bayes}
  to express

  \begin{align}\nonumber
    \mathbb{E}[f(v_0,\bar{x}_t)|\mathcal{Y}_t]&=\frac{\int_{\mathbb{R}^m}\pi(dx)\int_{\mathbb{R}^{2m}}f(x,r_1) e^{-\frac{1}{2}x^TM_t x + x^T r_2} \eta_t (dr_1,dr_2)}{\int_{\mathbb{R}^m}\pi(dx)\int_{\mathbb{R}^{2m}} e^{-\frac{1}{2}x^TM_t x + x^T r_2} \eta_t (dr_1,dr_2)}\\\nonumber
                                              &=\frac{\int_{\mathbb{R}^m}e^{\frac{1}{2}x^T(Q_t-M_t) x + x^T \tilde{b}_t}\pi(dx)\int_{\mathbb{R}^{2m}}f(x,r_1)  \tilde{\eta}_t (dr_1,dr_2)}{\int_{\mathbb{R}^m}e^{\frac{1}{2}x^T(Q_t-M_t) x + x^T \tilde{b}_t}\pi(dx)\int_{\mathbb{R}^{2m}}  \tilde{\eta}_t (dr_1,dr_2)}\\\label{test}
                                              &=\frac{\int_{\mathbb{R}^m}e^{\frac{1}{2}x^T(Q_t-M_t) x + x^T \tilde{b}_t}\pi(dx)\int_{\mathbb{R}^{2m}}f(x,r_1)  \tilde{\eta}_t (dr_1,dr_2)}{\int_{\mathbb{R}^m}e^{\frac{1}{2}x^T(Q_t-M_t) x + x^T \tilde{b}_t}\pi(dx)}
  \end{align}
  where $\tilde{\eta}_t$ is a Gaussian measure with mean
  $ \begin{pmatrix}
    \tilde{m}_t+ S_t x\\
    \tilde{b}_t +Q_t x
  \end{pmatrix}$ and covariance $ \begin{pmatrix}
    \tilde{P}_t & S_t \\
    S_t^T & Q_t
  \end{pmatrix}$. Setting
  $f(v_0, \bar{x}_t)= \tilde{f}(\Phi_tv_0 +\bar{x}_t)$, and taking
  $\gamma_t$ to be a Gaussian measure with mean $0$ and covariance
  $\tilde{P}_t$, we have
 
  \begin{align}
    \mathbb{E}[\tilde{f}(\Phi_tv_0 +\bar{x}_t)|\mathcal{Y}_t] &=\frac{\int_{\mathbb{R}^m}e^{\frac{1}{2}x^T(Q_t-M_t) x + x^T \tilde{b}_t}\pi(dx)\int_{\mathbb{R}^{2m}}\tilde{f}(\Phi_t x +r_1)  \tilde{\eta}_t (dr_1,dr_2)}{\int_{\mathbb{R}^m}e^{-\frac{1}{2}x^T(Q_t-M_t) x + x^T \tilde{b}_t}\pi(dx)} \nonumber \\
                                                              &=\frac{\int_{\mathbb{R}^m}e^{\frac{1}{2}x^T(Q_t-M_t) x + x^T \tilde{b}_t}\pi(dx)\int_{\mathbb{R}^{m}}\tilde{f}(\Phi_t x +\tilde{m}_t+ S_t x + r_3)  \gamma_t (dr_3) }{\int_{\mathbb{R}^m}e^{\frac{1}{2}x^T(Q_t-M_t) x + x^T \tilde{b}_t}\pi(dx)} \,. \label{eq:exptft}
 \end{align}
 Now setting $\tilde{f}(x)=x$ (this can be done even though
 $\tilde{f}$ is not bounded because $\tilde{f}$ is integrable with
 respect to Gaussian measure), we obtain the conditional mean as

 \begin{align*}
   \mathbb{E}[x_t|\mathcal{Y}_t] = \tilde{m}_t+ \mathbb{E}[(\Phi_t+ S_t )v_0|\mathcal {Y}_t]
   = \tilde{m}_t+ (\Phi_t+ S_t ) \mathbb{E}[v_0|\mathcal {Y}_t]
 \end{align*}
 Now observe from \eqref{S} that

 \begin{align*}
   \frac{d}{dt} {(\Phi_t+S_t)}=(A_t-\tilde{P}_tC_t ^T C_t)(\Phi_t+S_t),
 \end{align*}
 which has the same form as \eqref{hom}, and thus from Theorem
 \eqref{mainth} it follows that $\|\Phi_t +S_t \| \rightarrow 0$ as
 $t\rightarrow \infty$.
 
 \begin{align*}
   \|\mathbb{E}[x_t|\mathcal{Y}_t]- \tilde{m}_t\| &= \|(\Phi_t+ S_t )\mathbb{E}[v_0|\mathcal {Y}_t]\|\\
                                                  &\leq K_0 \|(\Phi_t+ S_t )\| \quad \mathbb{P}-a.s.\\
                                                  &\xrightarrow[]{t \to \infty} 0 \quad \mathbb{P}-a.s.\,, 
 \end{align*}
 because $\mathbb{E}[v_0|\mathcal{Y}_t]$ is uniformly integrable
 (square integrable, in particular). Therefore,

 \begin{align*}
   \mathbb{E}[x_t|\mathcal{Y}_t]-\tilde{m}_t\rightarrow 0 \;\;\mathbb{P}-a.s. \;\;and\;\; in\;\; L^2
 \end{align*}
 Now, if we can prove that
 $(X^{\bar{m},\bar{P}}_t-\tilde{m}_t)\rightarrow 0,\;\;
 \mathbb{P}-a.s$ then we shall have shown that

 \begin{align*}
   \mathbb{E}[x_t|\mathcal{Y}_t]-X^{\bar{m},\bar{P}}_t\rightarrow 0,\;\; \mathbb{P}-a.s
 \end{align*}
 To that end, consider

 \begin{align*}
   d(\tilde{m}_t-X^{\bar{m},\bar{P}}_t)&=(A_t-P^{\bar{P}}_t C_t^TC_t)(\tilde{m}_t-X^{\bar{m},\bar{P}}_t)dt+ (\tilde{P}_t-P^{\bar{P}}_t)C_t^T(dy_t-C_t \mathbb{E}[x_t|\mathcal{Y}_t])\\
                                       & + (\tilde{P}_t-P^{\bar{P}}_t)C_t^TC_t(\mathbb{E}[x_t|\mathcal{Y}_t]-\tilde{m}_t)dt
 \end{align*}
 whose solution can be expressed as

 \begin{align*}
   \tilde{m}_t-X^{\bar{m},\bar{P}}_t &= \Psi ^{\bar{P}}_t(\tilde{m}_0-\bar{X}_0) + \int_0 ^t \Psi ^{\bar{P}}_t\big(\Psi ^{\bar{P}}_s\big)^{-1}(\tilde{P}_s-P^{\bar{P}}_s)C_s^T (dy_s-C_s \mathbb{E}[x_s|\mathcal{Y}_s])\\ 
                                     &+ \int_0^t \Psi ^{\bar{P}}_t\big(\Psi ^{\bar{P}}_s\big)^{-1}(\tilde{P}_s-P^{\bar{P}}_s)C_s^TC_s(\mathbb{E}[x_s|\mathcal{Y}_s]-\tilde{m}_s)ds\\
                                     &= J_1 +J_2+J_3,
 \end{align*}
 where $J_1=\Psi ^{\bar{P}}_t(\tilde{m}_0-\bar{X}_0)$,
 $J_2=\int_0 ^t \Psi ^{\bar{P}}_t\big(\Psi
 ^{\bar{P}}_s\big)^{-1}(\tilde{P}_s-\bar{P}_s)C_s^T (dy_s-C_s
 \mathbb{E}[x_s|\mathcal{Y}_s])$, and
 $J_3=\int_0^t \Psi ^{\bar{P}}_t\big(\Psi
 ^{\bar{P}}_s\big)^{-1}(\tilde{P}_s-\bar{P}_s)C_s^TC_s(\mathbb{E}[x_s|\mathcal{Y}_s]-\tilde{m}_s)ds$.
 In view of Theorem \ref{mainth}, it is easy to check that
$J_1 \rightarrow 0$ and $J_2 \rightarrow 0\;\ \mathbb{P}-a.s$. Thus,
consider

\begin{align*}
  J_3&= \int_0^t \Psi ^{\bar{P}}_t\big(\Psi ^{\bar{P}}_s\big)^{-1}(\tilde{P}_s-P^{\bar{P}}_s)C_s^TC_s(\mathbb{E}[x_s|\mathcal{Y}_s]-\tilde{m}_s)ds\\
     &=\Psi ^{\bar{P}}_t(P'-\bar{P})  \int_0^t \big(\Psi ^{P'}_s\big)^TC_s^TC_s \Psi ^{P'}_s\mathbb{E}[v_0|\mathcal{Y}_s]ds\\
     &= \Psi ^{\bar{P}}_t(P'-\bar{P})  \int_0^t \big(\Psi ^{P'}_s\big)^TC_s^TC_s \Psi ^{P'}_s(\mathbb{E}[v_0|\mathcal{Y}_s]-\mathbb{E}[v_0|\mathcal{Y}_\infty])ds +\Psi ^{\bar{P}}_t(P'-\bar{P})  \int_0^t \big(\Psi ^{P'}_s\big)^TC_s^TC_s \Psi ^{P'}_s ds \, \mathbb{E}[v_0|\mathcal{Y}_\infty] \\
     &= L_1 +L_2,
\end{align*} 
where, 

\begin{align*}
  L_1 &=\Psi ^{\bar{P}}_t(P'-\bar{P})  \int_0^t \big(\Psi ^{P'}_s\big)^TC_s^TC_s \Psi ^{P'}_s(\mathbb{E}[v_0|\mathcal{Y}_s]-\mathbb{E}[v_0|\mathcal{Y}_\infty])ds\\
  L_2&=\Psi ^{\bar{P}}_t(P'-\bar{P})  \int_0^t \big(\Psi ^{P'}_s\big)^TC_s^TC_s \Psi ^{P'}_s ds \, \mathbb{E}[v_0|\mathcal{Y}_\infty] 
\end{align*}
Again, using the uniform bound on $C_s$ and theorem~\ref{mainth}, it
is clear that $L_2 \rightarrow 0\;\ \mathbb{P}-a.s$. In order to show
that $L_1 \rightarrow 0 \;\ \mathbb{P}-a.s$, It suffices to show that
$\|\int_0^t \big(\Psi ^{P'}_s\big)^TC_s^TC_s \Psi
^{P'}_s(\mathbb{E}[v_0|\mathcal{Y}_s]-\mathbb{E}[v_0|\mathcal{Y}_\infty])ds\|
<\infty$. To that end, we know that for a given $\epsilon>0$, there is
a $t_{\epsilon} >0$ such that for every $t>t_{\epsilon}$,
$\|\mathbb{E}[v_0|\mathcal{Y}_s]-\mathbb{E}[v_0|\mathcal{Y}_\infty]\|<\epsilon$.

\begin{align*}
  \left\|\int_0^t \big(\Psi ^{P'}_s\big)^TC_s^TC_s \Psi ^{P'}_s(\mathbb{E}[v_0|\mathcal{Y}_s]-\mathbb{E}[v_0|\mathcal{Y}_\infty])ds\right\|&\leq  \left\|\int_0^{t_{\epsilon}} \big(\Psi ^{P'}_s\big)^TC_s^TC_s \Psi ^{P'}_s(\mathbb{E}[v_0|\mathcal{Y}_s]-\mathbb{E}[v_0|\mathcal{Y}_\infty])ds \right\|\\
                                                                                                                                           & + \left\| \int_{t_{\epsilon}}^t \big(\Psi ^{P'}_s\big)^TC_s^TC_s \Psi ^{P'}_s(\mathbb{E}[v_0|\mathcal{Y}_s]-\mathbb{E}[v_0|\mathcal{Y}_\infty])ds \right\|\\
                                                                                                                                           &< \infty\;\;\textrm{uniformly in $t$, by \eqref{integrability}}
\end{align*}
Therefore, the above calculation implies that
$(X^{\bar{m},\bar{P}}_t-\tilde{m}_t)\rightarrow 0,\;\;
\mathbb{P}-a.s$. This concludes the proof of \eqref{1s}.

We again follow the method of \cite{ocone1996asymptotic} to next prove
\eqref{2s}. To this end, consider the optimal filtering distribution
$\pi _t$ (recall
$\pi_t(B)=\mathbb{E}[\mathsf{1}_{x_t \in B}|\mathcal{Y}_t]$) and the
Gaussian $\tilde{\mu} _t:=\mathcal{N}(\tilde{m}_t,\tilde {P}_t)$. For
a bounded uniformly continuous function
$g:\mathbb{R}^n \rightarrow \mathbb{R}$, using the expression from
\eqref{eq:exptft},

\begin{align}
  \int_{\mathbb{R}^n} g(x) \pi_t(dx)-&\int_{\mathbb{R}^n} g(x)
                                       \tilde{\mu} _t (dx) \nonumber \\
  &=\frac{\int_{\mathbb{R}^m}e^{\frac{1}{2}x^T(Q_t-M_t) x + x^T \tilde{b}_t}\pi(dx)\int_{\mathbb{R}^{m}}g(\Phi_t x +\tilde{m}_t+ S_t x + r_3)  \gamma_t (dr_3) }{\int_{\mathbb{R}^m}e^{\frac{1}{2}x^T(Q_t-M_t) x + x^T \tilde{b}_t}\pi(dx)}-\int_{\mathbb{R}^n} g(x) \tilde{\mu}_t (dx) \nonumber \\
  \label{unic} &=\frac{\int_{\mathbb{R}^m}e^{\frac{1}{2}x^T(Q_t-M_t) x + x^T \tilde{b}_t}\pi(dx)\int_{\mathbb{R}^{m}}\left[g(\Phi_t x +\tilde{m}_t+ S_t x + r_3)-g(\tilde{m}_t+ S_t r_3) \right] \gamma_t (dr_3) }{\int_{\mathbb{R}^m}e^{\frac{1}{2}x^T(Q_t-M_t) x + x^T \tilde{b}_t}\pi(dx)} \,,
\end{align}
where the last line is obtained by using definition of $\gamma_t(dx)$
and by multiplying and the second term in the second line above by
$\int_{\mathbb{R}^m}e^{\frac{1}{2}x^T(Q_t-M_t) x + x^T
  \tilde{b}_t}\pi(dx)$. Now, if we partition the $\pi(dx)$ integral
into regions $|(\Phi_t +S_t)x|<\delta$ and
$|(\Phi_t +S_t)x|\geq\delta$ for a fixed $\delta >0$, then

\begin{align*}
  \int_{\mathbb{R}^n} g(x) \pi_t(dx)-&\int_{\mathbb{R}^n} g(x) \tilde{\mu}_t (dx)\\&\leq \sup_{|z_1-z_2|<\delta }{|g(z_1)-g(z_2)|}+ 2\sup_{z}{(g(z))}\mathbb{E}[\mathsf{1}_{|(\Phi_t+S_t)x_0|>\delta}|\mathcal{Y}_t]\\
                                     &\leq \sup_{|z_1-z_2|<\delta }{|g(z_1)-g(z_2)|}+ \frac{2\sup_{z}{(g(z))}}{\delta^2}\|\Phi_t +S_t\|^2 \mathbb{E}[|x_0|^2]\\
                                     &\leq  \sup_{|z_1-z_2|<\delta }{|g(z_1)-g(z_2)|} \;\;\textrm{as}\;\; t\to \infty 
\end{align*}
where the second inequality follows from Chebyshev's
inequality. Observe now that for any $\delta_0 >0$, we can choose
sufficiently small $\delta$, such that
$\sup_{|z_1-z_2|<\delta }{|g(z_1)-g(z_2)|}<\delta_0$ which implies
that $\pi_t$ converges weakly to $\tilde{\mu}_t \,\, \mathbb{P}-a.s.$
as $t\to\infty$. Using now the fact that
$(X^{\bar{m},\bar{P}}_t-\tilde{m}_t)\rightarrow 0,\;\; \mathbb{P}-a.s$
and $(P^{\bar{P}}_t-\tilde{P}_t)\rightarrow 0$, we conclude that
$[\pi_t(g)-\mathcal{N}(X^{\bar{m},\bar{P}}_t,P^{\bar{P}}_t)(g)]\xrightarrow[]
{t \to \infty} 0,\;\; \mathbb{P}-a.s$.
\end{proof}
\begin{remark}
  We have shown that the optimal filter is asymptotically proximal to
  the Gaussian distribution with mean and covariance given by
  solutions of \eqref{eq2} and \eqref{DRE} for arbitrary initial
  conditions.  In contrast to the results in
  \cite{ocone1996asymptotic}, our methods are not sufficient to prove
  the exponential convergence in our case of zero system noise.
\end{remark}

\section{Small noise analysis}\label{S5}

In this section, we would like to study the small system noise
behaviour of linear filter. We initialise the system with noise and
zero noise system with same initial conditions and study the behaviour
of both solutions for same set of observations. Consider the processes
given by,

\begin{align}\label{epssystem}
  x^{\epsilon}_t&=x^{\epsilon}_{0}+\int_{0} ^t A_sx^{\epsilon}_sds +\epsilon\int_0 ^t F_s dV_s^{\epsilon},\\ \label{epsobserv}
  y^{\epsilon}_t&=\int_{0} ^tC_sx^{\epsilon}_sds+\int_{0} ^t R^{\frac{1}{2}}_s dW_s ^{\epsilon},\\\nonumber
  x^{\epsilon}_{0}&\sim\mathcal{N}(m_0,P_0)\;\; t\geq 0,
\end{align}
where, $V_t^{\epsilon}$ and $W_t^{\epsilon}$ are mutually independent
$\mathcal{F}_t$-standard Brownian motion. Also, $x_0$ is mutually
independent with respect to $V_t^{\epsilon}$ and $W_t^{\epsilon}$.

\begin{align}
  d\hat{x}^{\epsilon}_t&=A_t\hat{x}^{\epsilon}_tdt+Q^{\epsilon}_t C^T _t R^{-1}_t(dy^{\epsilon}_t-C_t\hat{x}^{\epsilon}_tdt), \\
  \dot{Q}^{\epsilon}_t&= A_tQ^{\epsilon}_t +Q^{\epsilon}_tA^T _t -Q^{\epsilon}_tC^T _t R^{-1} _t C_tQ^{\epsilon}_t +\epsilon ^2 F_t F_t ^T,\;\ \\
  \hat{x}^{\epsilon}_{0}&=m,  \;\;\ Q^{\epsilon}_{0}=Q,
\end{align}   
where,
$\hat{x}^{\epsilon}_t=\mathbb{E}[x^{\epsilon}_t|\mathcal{Y}^{\epsilon}_t]$
with
$\mathcal{Y}^{\epsilon}_t:=\sigma\{ y^{\epsilon}_s:0\leq s\leq t \}]$
and
$Q^{\epsilon}_t=\mathbb{E}[(x^{\epsilon}_t-\hat{x}^{\epsilon}_t)(x^{\epsilon}_t-\hat{x}^{\epsilon}_t)^T]$. We also define the new process $\hat{x}^0_t$ as
\begin{align*}
d\hat{x}^{0}_t&=A_t\hat{x}^{0}_tdt+P^Q_t C^T _t R^{-1}_t(dy^{\epsilon}_t-C_t\hat{x}^{0}_tdt).
\end{align*} 
Note that above definition involves $y^{\epsilon}_t$, instead of $y^0_t$. To proceed further, additional assumption is made in this analysis.
\begin{asu}\label{expo}
  $F_t$ is uniformly bounded in $t$ and
  $\dot{z}_t=(A_t-P^{Q}_tC_t^T R_t^{-1} C_t)z_t$ is exponentially
  stable.
\end{asu}
Sufficient conditions for the required exponential stability are given
in \cite{ni2016stability}.
\begin{thm}
  If $[A_t,C_t]$ are uniformly completely observable and assumption
  \eqref{expo} holds, then
  \begin{align*}
  \mathbb{P}\big(\lim _{\epsilon\to
    0}||\hat{x}^{\epsilon}_t-\hat{x}^0_t||=0, \forall t\geq 0\big)=1
\end{align*}   
  
\end{thm}
\begin{proof}
  Let us begin with observing

  \begin{align}
    & \frac{d}{dt}(Q^{\epsilon}_t-{P}^{Q}_t)\\
    &= B^Q _t(Q^{\epsilon}_t-{P}^{Q}_t) +(Q^{\epsilon}_t-{P}^{Q}_t)\big(B^Q _t\big)^T -(Q^{\epsilon}_t-{P}^{Q}_t)C^T _t R^{-1} _t C_t(Q^{\epsilon}_t-{P}^{Q}_t) +\epsilon ^2 F_t F_t ^T,\;\ Q^{\epsilon}_0-{P}^{Q}_0=0,
  \end{align}
  where, $B^Q _t:=A_t-P^{Q}_tC_t^T R_t^{-1} C_t$. If we define
  $\Delta P _t :=Q^{\epsilon}_t-{P}^{Q}_t$ and $\Psi ^Q _t$ is such
  that $\dot{\Psi }^Q_t= B^Q_t \Psi^Q_t,\;\;\Psi^Q_0=I$,

  \begin{align*}
    \frac{d}{dt}[\big(\Psi^Q_t\big)^{-1} \Delta P_t \big(\Psi^Q_t\big)^{-T}]&=\big(\Psi^Q_t\big)^{-1}(-\Delta P_t C^T _t R^{-1} _t C_t\Delta P_t +\epsilon ^2 F_t F_t ^T ) \big(\Psi^Q_t\big)^{-T}\\
    \Delta P_t &=\Psi^Q_t \int_0 ^t \big(\Psi^Q_s\big)^{-1}(-\Delta P_s C^T _s R^{-1} _s C_s\Delta P_s +\epsilon ^2 F_s F_s ^T ) \big(\Psi^Q_s\big)^{-T}ds \big(\Psi^Q_t\big)^T\\
                                                                            &\leq \epsilon ^2 \Psi^Q_t \int_0 ^t \big(\Psi^Q_s\big)^{-1} F_s F_s ^T  \big(\Psi^Q_s\big)^{-T}ds \big(\Psi^Q_t\big)^T
  \end{align*}
  From the assumption of exponential stability, we have
  $\|\Psi^Q_t \big(\Psi^Q_s\big)^{-1}\|\leq K e ^{-\alpha (t-s)}$, for
  some $K$, $\alpha$ and for all $t\geq s\geq 0$. Therefore,

  \begin{align*}
    0\leq\|\Delta P_t\|\leq \frac{\epsilon ^2 K F}{2\alpha}
  \end{align*}
  Now, we consider the evolution equation for
  $\hat{x}^{\epsilon}_t-\hat{x}^0_t$

  \begin{align*}
    d(\hat{x}^{\epsilon}_t-\hat{x}^0_t)&=B^Q_t (\hat{x}^{\epsilon}_t-\hat{x}^0_t)dt + (\Delta P_t)C^T _t R^{-1}_t (dy^{\epsilon}_t-C_t \hat{x}^{\epsilon}_tdt),\;\; \hat{x}_0 -\hat{x}^{\epsilon}_0=0,\\
    \hat{x}^{\epsilon}_t-\hat{x}^0_t&= \int_0 ^t \Psi^Q _t \big(\Psi^Q_s\big)^{-1} (\Delta P_s)C^T _s R^{-1}_s (dy^{\epsilon}_s-C_s \hat{x}^{\epsilon}_sds)\\
  \end{align*}
  Define,
  $u_t:= \int_0 ^t \big(\Psi^Q_s\big)^{-1} (\Delta P_s)C^T _s R^{-1}_s
  (dy^\epsilon _s-C_s \hat{x}^{\epsilon}_sds) $ and
  $\mathcal{B}_t:=\sigma \{y^\epsilon _r-\int_0^r C_s \hat{x}^{\epsilon}_sds :
  0\leq r \leq t \}$. Clearly, $u_t$ is a
  $\mathcal{B}_t$-martingale. For any given $\bar{T}\geq 0$ and
  $\lambda>0$, applying Doob's inequality to submartingale, $|u_t|$,
  we have

  \begin{align*}
    \mathbb{P}\big(\sup_{0\leq t\leq \bar{T}}|u_t|\geq\lambda \big)&\leq \frac{\mathbb{E}[|u_{\bar{T}}|]}{\lambda}\\
    \mathbb{P}\big(\sup_{0\leq t\leq \bar{T}}|\big(\Psi^Q_t\big)^{-1}(\hat{x}^{\epsilon}_t-\hat{x}^0_t)|\geq\lambda \big)&\leq \frac{\mathbb{E}[|\big(\Psi^Q_{\bar{T}}\big)^{-1}(\hat{x}^{\epsilon}_{\bar{T}}-\hat{x}^0_{\bar{T}})|]}{\lambda}\leq \|\big(\Psi^Q_{\bar{T}}\big)^{-1}\| \frac{\epsilon \sqrt{KFM}}{2\alpha \lambda}\\
    \mathbb{P}\big(\|\big(\Psi^Q_{\bar{T}}\big)^{-1}\|\sup_{0\leq t\leq \bar{T}}|(\hat{x}^{\epsilon}_t-\hat{x}^0_t)|\geq \lambda\big)&\leq \|\big(\Psi^Q_{\bar{T}}\big)^{-1}\| \frac{\epsilon \sqrt{KFM}}{2\alpha \lambda}s
  \end{align*}
  Now, choose $\epsilon= \frac{1}{n^2},\;\; n\in \mathbb{N}$ and
  $\lambda = \lambda_0 K e^ {\alpha \bar{T}}$ (arbitrariness of
  $\lambda$ is now in $\lambda_0$),

  \begin{align*}
    \mathbb{P}\big(K e^ {\alpha \bar{T}}\sup_{0\leq t\leq \bar{T}}|(\hat{x}^{\epsilon}_t-\hat{x}^0_t)|\geq \lambda_0 K e^ {\alpha \bar{T}}\big)&\leq K e^ {\alpha \bar{T}} \frac{ \sqrt{KFM}}{2 n^2 \alpha \lambda_0K e^ {\alpha \bar{T}}}\\
    \mathbb{P}\big(\sup_{0\leq t\leq \bar{T}}|(\hat{x}^{\epsilon}_t-\hat{x}^0_t)|\geq \lambda_0 \big)&\leq  \frac{\epsilon \sqrt{KFM}}{2\alpha n^2 \lambda_0}
  \end{align*}
  Then, using Borel-Cantelli lemma, we conclude that
  $\mathbb{P}\big(\lim _{n\to
    \infty}|\hat{x}^{\frac{1}{n^2}}_t-\hat{x}^0_t|=0, \forall
  \bar{T}\geq t\geq 0\big)=1, \; \forall \bar{T}\geq 0$. Therefore, the small noise
  limits are non-singular in the case of Gaussian initial condition.
\end{proof}
\section*{Acknowledgements}

ASR and AA would like to thank Amarjit Budhiraja and Eric S.~Van Vleck for valuable discussions. ASR and AA would also like to thank The Statistical and Applied Mathematical Sciences Institute (SAMSI), Durham, NC, USA where a part of the work was completed. ASR's visit to SAMSI was supported by Infosys Foundation Excellence Program of ICTS. SV and AA acknowledge generous support of the AIRBUS Group Corporate Foundation Chair in Mathematics of Complex Systems.
\bibliographystyle{siam}
\bibliography{reddy2018asymptotic}

\end{document}